\documentclass[12pt, reqno]{amsart}
\usepackage{enumerate,amsmath,amssymb,bm,ascmac,amsthm,url}
\usepackage{tikz}
\usepackage[margin=30truemm]{geometry}
\usepackage{here}

\theoremstyle{definition}
\newtheorem{thm}{Theorem}[section]
\newtheorem{thm*}{Theorem}
\newtheorem{defi}[thm]{Definition}
\newtheorem{defi*}{Definition}
\newtheorem{lem}[thm]{Lemma}
\newtheorem{lem*}{Lemma}
\newtheorem{pro}[thm]{Proposition}
\newtheorem{pro*}{Proposition}

\newtheorem{cor}[thm]{Corollary}

\newcommand{\MC}[1]{\mathcal{#1}}
\newcommand{\MB}[1]{\mathbb{#1}}

\newcommand{\Span}[1]{\langle #1 \rangle}

\newcommand{\NH}{\tilde{H}}

\DeclareMathOperator{\Spec}{Spec}
\DeclareMathOperator{\re}{Re}
\DeclareMathOperator{\Tr}{Tr}

\usepackage{tikz}
\usetikzlibrary{backgrounds}
\tikzset{v/.style = {
circle, fill = black, inner sep = 0.8mm, line width = 1pt}}

\makeatletter

\@addtoreset{equation}{section}
\makeatother

\title[Quantum walks defined by digraphs]{Quantum walks defined by digraphs and generalized Hermitian adjacency matrices}

\author[S. Kubota]{Sho Kubota$^1$}
\thanks{$^1$S.K. is supported by JSPS KAKENHI (Grant No. 20J01175).}
\author[E. Segawa]{Etsuo Segawa$^2$}
\thanks{$^2$E.S. acknowledges financial support from the Grant-in-Aid for Scientific Research (C) 19K03616, 
and Research Origin for Dressed Photon.}
\author[T. Taniguchi]{Tetsuji Taniguchi$^3$}
\thanks{$^3$T.T.  was supported by JSPS KAKENHI (Grant No. 16K05263).}
\address{Department of Applied Mathematics, Faculty of Engineering, Yokohama National University,
Hodogaya, Yokohama 240-8501, Japan}
\email{kubota-sho-bp@ynu.ac.jp}
\address{Graduate School of Education Center and Graduate School of Environment and Information Sciences, Yokohama National University,
Hodogaya, Yokohama 240-8501, Japan}
\email{segawa-etsuo-tb@ynu.ac.jp}
\address{Department of Electronics and Computer Engineering,
Hiroshima Institute of Technology, Hiroshima 731-5193, Japan}
\email{t.taniguchi.t3@cc.it-hiroshima.ac.jp}

\date{}

\begin{document}
\maketitle
\begin{abstract}
We propose a quantum walk defined by digraphs (mixed graphs).
This is like Grover walk that is perturbed by
a certain complex-valued function defined by digraphs.
The discriminant of this quantum walk is a matrix
that is a certain normalization of generalized Hermitian adjacency matrices.
Furthermore,
we give definitions of the positive and negative supports of the transfer matrix
and exhibit explicit formulas of supports of their square.
Also, we provide tables on the identification of digraphs by their eigenvalues.
\end{abstract}

\section{Introduction}

Since the beginning of the millennium,
quantum walks have been actively studied not only in mathematics but also in physics.
Several papers and books on the implementation of the experimental setups of them
have also been published \cite{FSS, GS, MW, NP},
and quantum walks on networks are reviewed in \cite{BDF}.
Quantum walks on digraphs are also studied as applications of topological phase's simulators,
see for example \cite{EKOS}\footnote{
This dynamics of the quantum walk treated on a half plain of the two-dimensional lattice \cite{EKOS}
can be regarded as a quantum walk on a digraph,
which is not symmetric,
because a quantum walker moves to $x$ and $y$ directions, alternatively.
}.
It can be said that quantum walks are studied in a very wide range of contexts.
Due to the wide range of them,
fundamental research such as proposing a model defined by digraphs
is important for future development of application to quantum walks.
In this paper,
we introduce a quantum walk model which can be defined not only undirected graphs but also digraphs.
Moreover, it allows us realistic eigenvalue analysis.

The Grover walks are the simplest quantum walks to be defined by providing graphs,
and they are recently well-studied,
including topics on quantum search \cite{Portugalbook}, 
graph isomorphism problem~\cite{EHSW}, 
periodicity~\cite{HKSS2017, KSTY2018, Y2017, Y2019},
relation to the Ihara Zeta function~\cite{KS2011, KSS2019, RenETAL}, and so on.
They are originally defined by undirected graphs,
but we will extend them to digraphs.

In order to define our quantum walk,
we introduce a certain complex-valued function $\theta$ and generalize the shift operator,
which can be regarded as perturbations by $\theta$.
We define such function $\theta$ from digraphs and
consider something like Grover walk equipped with $\theta$.
This quantum walk can be seen as the one to have information on digraphs.
Moreover, the discriminant of our quantum walk (whose definition is given in (\ref{R01}))
coincides with a matrix that is a certain normalization of the Hermitian adjacency matrix,
which has been introduced in the context of spectral graph theory \cite{GM, LL}.

As Emms, Hancock, Severini and Wilson \cite{EHSW}
attempted to distinguish strongly regular graphs
by their spectrum of matrices coming from quantum walks,
we also want to use matrices coming from quantum walks to identify digraphs.
For this purpose,
we exhibit explicit formulas for the positive and negative supports of the square
of the transfer matrix of our quantum walk.
In the case of undirected regular graphs,
the eigenvalues of the positive support of the square of the Grover transfer matrices
are determined by the eigenvalues of the adjacency matrices \cite{GG}.
However, we will see that this is not always true for digraphs.
Also, we provide tables on the identification of digraphs by their eigenvalues.

This paper is organized as follows.
In Section~2,
we review basic terminologies related to graphs and quantum walks,
and define our quantum walk.
In Section~3, 
we give a definition of generalized Hermitian adjacency matrices,
and state that our quantum walk and these matrices are related to each other.
In Section~4,
we show so-called spectral mapping theorem,
that is,
we show that the eigenvalues of the transfer matrix can be written
roughly by those of generalized Hermitian adjacency matrices.
In Section~5,
we calculate the eigenvalues of the transfer matrix of the digraphs
cospectral with $K_n$ in the sense of Hermitian adjacency matrices
found by Guo and Mohar \cite{GM}.
In fact, it can be seen that
these digraphs are also cospectral in the sense of our transfer matrix.
In Section~6,
we give a definition of the positive and negative supports
of our transfer matrix to the $n$-th power.
For the case of $n = 2$,
we show that these matrices can be roughly expressed by
the ones of the Grover transfer matrix on the underlying graph.
In Section~7,
we confirm that Guo--Mohar's digraphs,
which could not be identified even by the eigenvalues of our transfer matrix,
can be partially identified by the supports of the square.
In Section~8,
consideration on the identification of digraphs
by the eigenvalues of matrices newly obtained in this study is given.
In this consideration, tables similar to the ones for adjacency matrices and Hermitian adjacency matrices
made by Guo and Mohar \cite{GM} are given.

\subsection{Histories of quantum walks on digraphs}

In this subsection,
we review previous works on quantum walks on digraphs.
In general,
it is difficult to define quantum walks on digraphs themselves.
Historically, Severini \cite{S02} focused on 1-factors of digraphs
and first defined coined quantum walks on digraphs.
In 2003, the following year,
Severini \cite{S03} considered a condition called strong quadrangularity
and showed that a line digraph is the digraph of a unitary matrix,
if and only if the digraph is Eulerian.
From the result,
we saw that quantum walks on digraphs can be defined when the digraph is Eulerian.
After that,
Acevedo and Gobron \cite{AG} defined quantum walks on Cayley graphs.
The Cayley graphs can be digraphs depending on connection sets,
so in this sense,
they also defined quantum walks on digraphs.
In 2007,
Montanaro \cite{M} successfully explained a necessary and sufficient condition
for defining a discrete-time quantum walk on a digraph by a property called reversibility.
For digraphs that are not reversible,
a method called a partially quantum walk has been presented. 
On the other hand,
Hoyer and Meyer \cite{HM} gave the first example of faster transport with a quantum walk on digraphs.
Other recent studies of quantum walks on digraphs are in \cite{CP, GL, L, Z}.

In contrast to previous studies,
we propose quantum walks defined by digraphs in another form.
Our idea is to generalize the Grover walks.
Roughly speaking, we consider the generalized Grover walks on the underlying graphs of digraphs
and change the phase using information of the digraphs.
By this idea,
we can not only easily define quantum walks for any digraphs,
but also provide models to allow realistic eigenvalue analysis.
Indeed,
we explicitly reveal eigenvalues of the transfer matrices of quantum walks defined by digraphs
cospectral with the complete graphs.

\subsection{A remark for readers}
The transfer matrix of our quantum walks is defined by digraphs.
However, our quantum walk is not the one {\it on} digraphs, so readers should note it.

\section{Grover walks and their generalization}

First, we review several notations on graphs.
Let $G=(V(G), E(G))$ be a finite simple and connected graph
with the vertex set $V(G)$ and the edge set $E(G)$.
For an edge $uv \in E(G)$, the arc from $u$ to $v$ is denoted by $(u, v)$.
The origin and terminus vertices of $e=(u, v)$ are denoted by $o(e), t(e)$, respectively.
We express $e^{-1}$ as the inverse arc of $e$.
We define $A(G)=\{ (u, v), (v, u) | uv \in E(G) \}$,
which is the set of the symmetric arcs of $G$.

\subsection{Grover walks on undirected graphs}

Let $G$ be a finite simple and connected graph.
Then, $G$ defines the following three complex matrices $K,S,C$.
First, $K = K(G)$ is the complex matrix whose rows are indexed by $V(G)$
and columns are indexed by $A(G)$, defined by
\[ K_{v,a} = \frac{1}{\sqrt{\deg t(a)}} \delta_{v,t(a)}, \]
where $\delta_{a,b}$ is the Kronecker delta symbol.
Note that $KK^* = I$.
Second, $S = S(G)$ is the complex matrix indexed by $A(G)$,
where
\[ S_{ab} = \delta_{a,b^{-1}}. \]
Third, $C = C(G)$ is the complex matrix indexed by $A(G)$, where
\[ C = 2K^*K-I. \]
The matrices $S$ and $C$ are
called the {\it shift operator} and the {\it coin operator}, respectively.
Usually, it is customary for these matrices $K, S, C$ to be defined
as mappings on $l^2$-spaces on $V(G)$ or $A(G)$,
but we would like to study quantum walks from the viewpoint of spectral graph theory,
so these operators are displayed in matrix and we will discuss by using matrices.

The {\it Grover transfer matrix} $U = U(G)$ is the product of
the shift operator and the coin operator, that is,
\[ U = SC. \]
Since $S$ and $C$ are unitary, $U$ is also a unitary matrix,
so $U$ can define a quantum walk on a given graph $G$.
We call this quantum walk the {\it Grover walk}.
Note that $U$ is the complex matrix indexed by $A(G)$
and the $(a,b)$-component can be calculated as
\begin{equation} \label{53}
U_{ab} = 
\frac{2}{\deg_{G} (t(b)) } \delta_{t(b), o(a)} - \delta_{a^{-1}, b},
\end{equation}
so some papers directly give a definition of the Grover transfer matrix by this computation.
In addition,
the {\it discriminant operator} $T$ of Grover walk is defined by
\begin{equation} \label{R00}
T = KSK^*,
\end{equation}
which plays an important role to analyze spectrum of $U$.

\subsection{Quantum walks defined by digraphs}

Based on the Grover walks,
our quantum walks will be defined.
Let $G$ be a finite simple and connected graph
and let $\theta : A(G) \to \MB{R}$ be a function.
The complex matrix $S_{\theta} = S_{\theta}(G)$ is the one indexed by $A(G)$, where
\[ (S_{\theta})_{ab} = e^{\theta(b)i}\delta_{a,b^{-1}}. \]
We consider a condition where $S_{\theta}$ is unitary and self-adjoint.
The following lemma is easy, but it is essentially important.


\begin{lem} \label{300}
{\it
With the above notations, we have the following:
\begin{enumerate}[(i)]
\item $S_{\theta}$ is unitary;
\item $S_{\theta}$ is self-adjoint if and only if
$\theta(a) + \theta(a^{-1}) \in 2\pi \MB{Z}$ for any $a \in A(G)$.
\end{enumerate}
}
\end{lem}

By this lemma,
the matrix $U_{\theta} = S_{\theta}C$ is also unitary,
where $C = C(G)$ is the coin operator given in the previous subsection.
We call $U_{\theta}$ the {\it transfer matrix} of a graph $G$.
In addition,
the {\it discriminant operator} $T_{\theta}$ of our quantum walk is defined by
\begin{equation} \label{R01}
T_{\theta} = KS_{\theta}K^*.
\end{equation}
As we will see later,
the discriminant of our quantum walk equipped with an appropriate function $\theta$
will be a meaningful matrix.
In this paper,
we discuss the quantum walks defined by $U_{\theta}$.

Now, we check several terminologies on digraphs.
A {\it digraph} $G$ consists of a finite set $V(G)$ of vertices
together with a subset $A(G) \subset V(G) \times V(G) \setminus \{ (x,x) \mid x \in V(G) \}$
of ordered pairs called {\it arcs}.
Define $A(G)^{-1} = \{ a^{-1} \mid a \in A(G) \}$ and $A(G)^{\pm 1} = A(G) \cup A(G)^{-1}$.
The graph $G^{\pm} = (V(G), A(G)^{\pm 1})$ is
so-called the {\it underlying graph} of a digraph $G$,
and this is often regarded as an undirected graph.
If $(x,y) \in A(G) \cap A(G)^{-1}$,
we say that the unordered pair $\{x,y\}$ is a {\it digon} of $G$.
For a vertex $x \in V(G)$,
define $\deg_{G} x = \deg_{G^{\pm}} x$.
A digraph $G$ is $k$-regular if $\deg_{G}x = k$ for any vertex $x \in V(G)$.
Throughout this paper, we assume that digraphs are weakly connected,
i.e.,
for any digraph $G$, we suppose that $G^{\pm}$ is connected.\footnote{
Quantum walks we will define are a generalization of the Grover walks.
As we will see later,
walkers in our quantum walk are on $G^{\pm}$.
In the case of Grover walks,
the connectedness of graphs is usually assumed,
so we need assume that $G^{\pm}$ is connected.
}

A function $\theta$ which satisfies the condition (ii) of Lemma~\ref{300}
can be defined from digraphs as follows.
Let $\eta \in \MB{R}$.
For a digraph $G = (V(G), A(G))$,
we define a function $\theta : A(G^{\pm}) \to \MB{R}$
by
\[
\theta(a) = \begin{cases}
\eta \qquad &\text{if $a \in A(G) \setminus A(G)^{-1}$,} \\
-\eta \qquad &\text{if $a \in A(G)^{-1} \setminus A(G)$,} \\
0 \qquad &\text{if $a \in A(G) \cap A(G)^{-1}$.}
\end{cases}
\]
Clearly,
\begin{equation} \label{301}
\theta(a) + \theta(a^{-1}) = 0
\end{equation}
holds for any $a \in A(G^{\pm})$.
We call this the {\it $\eta$-function} of a digraph $G$.
Then, we can think that a digraph $G$ equipped with an $\eta$-function $\theta$ defines
the transfer matrix $U_{\theta} = S_{\theta}(G^{\pm}) C(G^{\pm})$.
The operators (matrices) used in our quantum walks are summarized in Table~\ref{1000},
where $G = (V(G), A(G))$ is a digraph equipped with an $\eta$-function $\theta$.

\begin{table}[H]
  \centering
  \begin{tabular}{|c|c|c|c|}
\hline
Notation & Name & Indices of rows and columns & Definition \\
\hline
\hline
$K$ & Boundary & $V(G) \times A(G^{\pm})$ & $K_{v,a} =  \frac{1}{ \sqrt{\deg t(a)} } \delta_{v, t(a)}$ \\
\hline
$C$ & Coin & $A(G^{\pm}) \times A(G^{\pm})$ &$ C = 2K^*K - I$ \\
\hline
$S_{\theta}$ & Shift & $A(G^{\pm}) \times A(G^{\pm})$ & $(S_{\theta})_{ab} = e^{\theta(b)i}\delta_{a,b^{-1}}$ \\
\hline
$U_{\theta}$ & Transfer & $A(G^{\pm}) \times A(G^{\pm})$  & $U_{\theta} = S_{\theta} C$ \\
\hline
$T_{\theta}$ & Discriminant & $V(G) \times V(G)$ & $T_{\theta} = KS_{\theta}K^*$ \\ \hline
\end{tabular}
 \caption{The operators (matrices) used in our quantum walk} \label{1000}
\end{table}

\subsection{Example}

We provide an example.
As mentioned in the previous subsection,
our quantum walk is defined by a digraph and a real number $\eta$.
Let $G$ be a digraph in Figure~\ref{a01}.
We consider this digraph $G$ and $\eta = \pi/2$.
\begin{figure}[ht]
\begin{center}
\begin{tikzpicture}
[scale = 0.8,
line width = 1pt,
]
  \node[v] (1) at (-3, 0) {};
  \draw (-3,0.5) node {$v_1$};
  \node[v] (2) at (0, 0) {};
  \draw (0,0.5) node {$v_2$};
  \node[v] (3) at (3, 1.7) {};
  \draw (3.5,1.7) node {$v_3$};
  \node[v] (4) at (3, -1.7) {};
  \draw (3.5,-1.7) node {$v_4$};
  \draw[->] (1) to [bend left = 15] (2);
  \draw[->] (2) to [bend left = 15] (1);
  \draw (-1.5,0.6) node {$a$};
  \draw (-1.5,-0.6) node {$a^{-1}$};
  \draw[->] (2) to (4);
  \draw (1.5, -1.2) node {$c$};
  \draw[->] (4) to (3);
  \draw (3.5, 0) node {$d$};
  \draw[->] (3) to (2);
  \draw (1.5, 1.2) node {$b$};
\end{tikzpicture}
\caption{Digraph $G$} \label{a01}
\end{center}
\end{figure}
For the sets of indices $V(G) = \{ v_1, v_2, v_3, v_4 \}$ and
$A(G^{\pm}) = \{ a, a^{-1}, b, b^{-1}, c, c^{-1}, d, d^{-1}\}$,
we have
\begin{align*}
K &=
\begin{bmatrix}
1&0&0&0&0&0&0&0 \\
0&\frac{1}{\sqrt{3}}&\frac{1}{\sqrt{3}}&0&0&\frac{1}{\sqrt{3}}&0&0 \\
0&0&0&\frac{1}{\sqrt{2}}&0&0&\frac{1}{\sqrt{2}}&0 \\
0&0&0&0&\frac{1}{\sqrt{2}}&0&0&\frac{1}{\sqrt{2}}
\end{bmatrix}, \\
C = 2K^* K - I &=
\begin{bmatrix}
1&0&0&0&0&0&0&0 \\
0&-\frac{1}{3}&\frac{2}{3}&0&0&\frac{2}{3}&0&0 \\
0&\frac{2}{3}&-\frac{1}{3}&0&0&\frac{2}{3}&0&0 \\
0&0&0&0&0&0&1&0 \\
0&0&0&0&0&0&0&1 \\
0&\frac{2}{3}&\frac{2}{3}&0&0&-\frac{1}{3}&0&0 \\
0&0&0&1&0&0&0&0 \\
0&0&0&0&1&0&0&0
\end{bmatrix}, \\
S_{\theta} &=
\begin{bmatrix}
0&1&0&0&0&0&0&0 \\
1&0&0&0&0&0&0&0 \\
0&0&0&-i&0&0&0&0 \\
0&0&i&0&0&0&0&0 \\
0&0&0&0&0&-i&0&0 \\
0&0&0&0&i&0&0&0 \\
0&0&0&0&0&0&0&-i \\
0&0&0&0&0&0&i&0 
\end{bmatrix},
\end{align*}
so the transfer matrix is
\[
U_{\theta} = S_{\theta}C =
\begin{bmatrix}
0&-\frac{1}{3}&\frac{2}{3}&0&0&\frac{2}{3}&0&0 \\
1&0&0&0&0&0&0&0 \\
0&0&0&0&0&0&-i&0 \\
0&\frac{2}{3}i&-\frac{1}{3}i&0&0&\frac{2}{3}i&0&0 \\
0&-\frac{2}{3}i&-\frac{2}{3}i&0&0&\frac{1}{3}i&0&0 \\
0&0&0&0&0&0&0&i \\
0&0&0&0&-i&0&0&0 \\
0&0&0&i&0&0&0&0 
\end{bmatrix}.
\]
Unlike the conventional Grover walks,
our transfer matrix takes imaginary numbers.
It can be seen that these imaginary numbers give information of digraphs.

\section{Hermitian adjacency matrices}

Guo--Mohar \cite{GM} and Li--Liu \cite{LL} independently defined
the Hermitian adjacency matrix as a new matrix determined by a digraph.
Since this matrix is Hermitian,
the eigenvalues are real numbers
and the interlacing theorem can be used,
so this is a convenient matrix for estimating invariants of graphs.
Basic properties of the Hermitian adjacency matrix can be seen in both papers.
In \cite{GM}, we can see estimation on the maximum eigenvalue and the spectral radius.
In \cite{LL}, consideration of the Hermitian energy is carried out.

For a digraph $G$,
the {\it Hermitian adjacency matrix} $H = H(G)$ is
the complex matrix indexed by the vertex set $V(G)$, where
\[
H_{xy} = \begin{cases}
1 \qquad &\text{if $(x,y) \in A(G) \cap A(G)^{-1}$,} \\
i \qquad &\text{if $(x,y) \in A(G) \setminus A(G)^{-1}$,} \\
-i \qquad &\text{if $(x,y) \in A(G)^{-1} \setminus A(G)$,} \\
0 \qquad &\text{otherwise.}
\end{cases}
\]
When every edge of $G$ lies in a digon,
$H(G)$ coincides with the adjacency matrix,
so in this sense,
the Hermitian adjacency matrices can be seen
as a generalization of the ordinary adjacency matrices of undirected graphs.
We generalize this matrix by interpreting the complex value $i$ as $e^{\frac{\pi}{2}i}$.
For $\eta \in \MB{R}$ and a digraph $G$,
the {\it $\eta$-Hermitian adjacency matrix $H_{\eta} = H_{\eta}(G)$} is
the complex matrix indexed by the vertex set $V(G)$, where
\[
(H_{\eta})_{xy} = \begin{cases}
1 \qquad &\text{if $(x,y) \in A(G) \cap A(G)^{-1}$,} \\
e^{\eta i} \qquad &\text{if $(x,y) \in A(G) \setminus A(G)^{-1}$,} \\
e^{-\eta i} \qquad &\text{if $(x,y) \in A(G)^{-1} \setminus A(G)$,} \\
0 \qquad &\text{otherwise.}
\end{cases}
\]
Note that $H_{\frac{\pi}{2}} = H$.

Due to the spectral mapping theorem which will be described in Section 4,
the eigenvalues of $U_{\theta}$ are roughly expressed by
those of a matrix indexed by the vertex set, called discriminant.
In the case of Grover walks,
the discriminant operator corresponds to the transition matrix of the isotropic random walk.
As we will see, the discriminant of our quantum walk is
a certain normalized Hermitian adjacency matrix.
For a digraph $G$,
we define the {\it degree matrix} $D = D(G)$ by $D_{xy} = (\deg_{G}x)\delta_{x,y}$
for vertices $x,y \in V(G)$.
For $\eta \in \MB{R}$,
we define the {\it normalized $\eta$-Hermitian adjacency matrix} $\NH_{\eta}$ by
\[ \NH_{\eta} = D^{-\frac12} H_{\eta} D^{-\frac12}. \]
Note that if a digraph $G$ is $k$-regular,
then $D^{-\frac12} = \frac{1}{\sqrt{k}} I$,
so $\NH_{\eta} = \frac{1}{k} H_{\eta}$.
This implies that the eigenvalues of $\NH_{\eta}$ can be determined by $H_{\eta}$.

The following can be easily confirmed by calculating each components directly,
but this equality is valuable
in the sense of connecting quantum walks to spectral graph theory.

\begin{thm} \label{01}
{\it
Let $\eta \in \MB{R}$ and $G$ be a digraph equipped with the $\eta$-function $\theta$.
Then we have
\[ \NH_{\eta}(G) = KS_{\theta}K^*. \]
}
\end{thm}

\begin{proof}
Indeed,
\begin{align*}
(KS_{\theta}K^*)_{xy}
&= \sum_{a,b \in A(G^{\pm})}K_{x,a} (S_{\theta})_{ab} K_{y,b} \\
&= \sum_{\substack{a,b \in A(G^{\pm}) \\ t(a) = x \\ t(b) = y}} \frac{1}{\sqrt{\deg x}} \frac{1}{\sqrt{\deg y}}
e^{\theta(b)i} \delta_{a,b^{-1}} \\
&= \sum_{\substack{a \in A(G^{\pm}) \\ t(a) = x \\ o(a) = y}} \frac{1}{\sqrt{\deg x}} \frac{1}{\sqrt{\deg y}} e^{\theta(a^{-1})i} \\
&= \begin{cases}
\frac{e^{\theta((x,y))}}{\sqrt{\deg x}\sqrt{\deg y}} \quad &\text{if $x$ is adjacent to $y$ on $G^{\pm}$,} \\
0 &\text{otherwise.}
\end{cases} \\
&= \begin{cases}
\frac{1}{\sqrt{\deg x}\sqrt{\deg y}} \qquad &\text{if $(x,y) \in A(G) \cap A(G)^{-1}$,} \\
\frac{e^{\eta i}}{\sqrt{\deg x}\sqrt{\deg y}} \qquad &\text{if $(x,y) \in A(G) \setminus A(G)^{-1}$,} \\
\frac{e^{-\eta i}}{\sqrt{\deg x}\sqrt{\deg y}} \qquad &\text{if $(x,y) \in A(G)^{-1} \setminus A(G)$,} \\
0 \qquad &\text{otherwise.}
\end{cases} \\
&= \NH_{\eta}(G)_{xy}.
\end{align*}
\end{proof}

Supplementing a property of the eigenvalues of
the normalized $\eta$-Hermitian adjacency matrix,
we complete this section.
In general,
it is revealed that the eigenvalues of the discriminant are in the closed interval $[-1, 1]$
(See the proof of Proposition~1 in \cite{HKSS2014}),
but we give a proof of this fact from the viewpoint of spectral graph theory.

\begin{lem} \label{02}
{\it
Let $A \in M_n(\MB{C})$ be a complex matrix such that $A^* = A$ and $A^2 = A$.
Then $\Span{f, Af}$ is a real number for any vector $f \in \MB{C}^n$ and
\[ \Span{f,Af} \leq \Span{f,f} \]
holds.
}
\end{lem}

\begin{proof}
Pick a vector $f \in \MB{C}^n$.
From $A^* = A$ and $A^2 = A$,
\[ ||Af|| = \Span{Af,Af} = \Span{f, A^*Af} = \Span{f, A^2f} = \Span{f, Af}, \]
so $\Span{f, Af}$ is a real number.
Since the eigenvalues of $A$ are only $0,1$,
we can write as $f = g_0 + g_1$,
where $g_i$ is an eigenvector of $i \in \{0,1\}$.
Then we have
\begin{align*}
\Span{f,f} - \Span{f,Af} &= \Span{g_0+g_1, g_0+g_1} - \Span{g_0+g_1, g_1} \\
&= ||g_0||^2 + ||g_1||^2 - ||g_1||^2 \\
&= ||g_0||^2 \geq 0.
\end{align*}
\end{proof}

Remark that the matrix $A$ satisfying the conditions $A^2 = A$ and $A^* = A$ is known
as the orthogonal projection operator in the literature of the functional analysis.

\begin{pro} \label{03}
{\it
Let $\eta \in \MB{R}$ and $G$ be a digraph.
Then for any eigenvalue $\lambda$ of $\NH_{\eta}(G)$,
we have $|\lambda| \leq 1$.
}
\end{pro}
\begin{proof}
Let $f$ be an eigenvector of $\NH_{\eta} = \NH_{\eta}(G)$ of an eigenvalue $\lambda$.
Then
\begin{align*}
|\lambda|^2 ||f||^2 &= || \NH_{\eta} f ||^2 \\
&= || KS_{\theta}K^* f ||^2 \tag{by Theorem~\ref{01}} \\
&= \Span{KS_{\theta}K^*f, KS_{\theta}K^*f} \\
&= \Span{S_{\theta}K^*f, K^*KS_{\theta}K^*f} \\
&\leq \Span{S_{\theta}K^*f, S_{\theta}K^*f} \tag{by Lemma~\ref{02}} \\
&= \Span{K^*f, K^*f} \tag{$S_{\theta}$ is unitary} \\
&= \Span{f, KK^*f} \\
&= \Span{f, f} \tag{by $KK^* = I$} \\
&= ||f||^2,
\end{align*}
so we have the desired inequality.
\end{proof}

\section{Spectral mapping theorem}

The eigenvalues of the normalized $\eta$-Hermitian adjacency matrix are real,
and they are in the closed interval $[-1,1] \subset \MB{R}$ from Proposition~\ref{03}.
In fact, most eigenvalues of the transfer matrix $U_{\theta}$
are values obtained by lifting the eigenvalues of $\NH_{\eta}$ onto the unit circle in the complex plane.
Such facts are often called spectral mapping theorems,
and well-studied in quantum walks on graphs (see e.g., \cite{KSS2019} and its references therein).
In 2014,
Higuchi, Konno, Sato and Segawa \cite{HKSS2014} introduced certain quantum walks,
which are called {\it twisted Szegedy walks},
and revealed a spectral mapping theorem
derived from twisted random walks.
In this section,
we confirm that a spectral mapping theorem holds for our quantum walk via twisted Szegedy walks.

For a digraph $G$,
a {\it path} of $G$ is a sequence of arcs $(a_1, a_2, \dots, a_r)$
with $t(a_j) = o(a_{j+1})$ for $j = 1, \dots, r-1$.
If $t(a_r) = o(a_1)$, then it is said to be {\it closed}.
We denote the set of all closed paths of $G$ by $\MC{C}(G)$.
A function $A(G^{\pm}) \to (0,1]$ is called a {\it transition probability} if
\[ \sum_{\substack{a \in A(G^{\pm}) \\ o(a) = u}} p(a) = 1 \] for all $u \in V(G)$.
A transition probability $p$ is said to be {\it reversible}
if there exists a positive valued function $m : V(G) \to (0, \infty)$ such that
\[ m(o(a))p(a) = m(t(a))p(a^{-1}) \]
for all $a \in A(G^{\pm})$.

\begin{lem} \label{91}
{\it
Let $G$ be a digraph.
Then the transition probability $p$ defined by
\[ p(a) = \frac{1}{\deg_{G^{\pm}}o(a)} \]
is reversible.
}
\end{lem}

\begin{proof}
Considering the positive valued function $m$ defined by
\[ m(u) = \deg_{G^{\pm}}u, \]
we can check $m(o(a))p(a) = 1 = m(t(a))p(a^{-1})$ for any $a \in A(G)^{\pm 1}$.
\end{proof}

\subsection{Twisted Szegedy walks}

In this subsection,
we review twisted Szegedy walks.
Let $G$ be a finite simple and connected graph.
A function $w : A(G) \to \MB{C} \setminus \{0\}$ is a {\it weight function} if
\[ \sum_{\substack{a \in A(G)\\ o(a) = x}} |w(a)|^2 = 1 \]
for any $x \in V(G)$.
A function $\theta : A(G) \to \MB{R}$ is a {\it 1-form function} if
\[ \theta(a^{-1}) = - \theta(a) \]
for any $a \in A(G)$.
The operators (matrices) related to twisted Szegedy walks are summarized in Table~\ref{320}.

\begin{table}[H]
  \centering
  \begin{tabular}{|c|c|c|c|}
\hline
Notation & Name & Indices of rows and columns & Definition \\
\hline
\hline
$L$ & Boundary & $V(G) \times A(G)$ & $L_{v,a} = w(a) \delta_{v, o(a)}$ \\
\hline
$C^{(w)}$ & Coin & $A(G) \times A(G)$ &$ C^{(w)} = 2L^*L - I$ \\
\hline
$S^{(\theta)}$ & Twisted shift & $A(G) \times A(G)$  & $S^{(\theta)}_{ab} = e^{\theta(b)i}\delta_{a,b^{-1}}$ \\
\hline
$U^{(w, \theta)}$ & Time evolution & $A(G) \times A(G)$  & $U^{(w, \theta)} = S^{(\theta)} C^{(w)}$ \\
\hline
$T^{(w, \theta)}$ & Discriminant & $V(G) \times V(G)$ & $T^{(w, \theta)} = L S^{(\theta)} L^*$ \\ \hline
\end{tabular}
 \caption{The operators related to twisted Szegedy walks} \label{320}
\end{table}

Note that the twisted shift operator $S^{(\theta)}$ is unitary and self-adjoint.

Let $\varphi$ be a function from the unit circle $\{z \in \MB{C} \mid |z| = 1\}$
to the closed interval $[-1,1] \subset \MB{R}$, where $\varphi(z) = (z+z^{-1})/2$.

\begin{thm}[\cite{HKSS2014}] \label{310}
{\it
Let $G$ be a graph equipped with a weight function $w$ and a $1$-form function $\theta$.
Denote $m_{\pm 1}$ by the multiplicities of the eigenvalues $\pm 1$ of $T^{(w, \theta)}$,
respectively.
Then we have
\[ \Spec(U^{(w, \theta)}) = \varphi^{-1}(\Spec(T^{(w, \theta)}))
\cup \{1\}^{M_1} \cup \{-1\}^{M_{-1}}, \]
where $M_{\varepsilon} = \max\{0, |E(G)| - |V(G)| + m_{\varepsilon}\}$
for $\varepsilon \in \{\pm 1\}$.
}
\end{thm}

\subsection{Spectral mapping theorem for our quantum walk}

The function $w_0 : A(G) \to \MB{C} \setminus \{0\}$ defined by
\[ w_0(a) = \frac{1}{\sqrt{\deg o(a)}} \]
is clearly a weight function,
and any $\eta$-function $\theta$ is a $1$-form function by (\ref{301}).

\begin{thm} \label{21}
{\it
Let $G$ be a digraph equipped with an $\eta$-function $\theta$ and
$U_{\theta}$ be the transfer matrix.
Denote $m_{\pm 1}$ by the multiplicities of the eigenvalues $\pm 1$ of $\NH_{\eta}(G)$,
respectively.
Then we have
\begin{equation} \label{a30}
\Spec(U_{\theta}) = \varphi^{-1}(\Spec(\NH_{\eta}(G)))
\cup \{1\}^{M_1} \cup \{-1\}^{M_{-1}},
\end{equation}
where $M_{\varepsilon} = \max\{0, |E(G^{\pm})| - |V(G^{\pm})| + m_{\varepsilon}\}$
for $\varepsilon \in \{\pm 1\}$.
}
\end{thm}

\begin{proof}
We consider the twisted Szegedy walk on $G^{\pm}$ associated with $w_0$ and $\theta$.
By Theorem~\ref{310},
it is enough to show that $\NH_{\eta}(G) = T^{(w_0, \theta)}$ and
$U_{\theta}$ is similar to $U^{(w_0, \theta)}$.
The following equalities can be checked immediately:
\begin{enumerate}[(i)]
\item $L = K(G^{\pm}) S(G^{\pm})$;
\item $S^{(\theta)} = S_{\theta}(G^{\pm})$;
\item $S(G^{\pm}) S_{\theta}(G^{\pm}) S(G^{\pm}) = S_{\theta}(G^{\pm})$.
\end{enumerate}
Thus, as for the discriminant operators, we have
\begin{align*}
T^{(w_0, \theta)} &= L S^{(\theta)} L^* \\
&= K(G^{\pm}) S(G^{\pm}) S_{\theta}(G^{\pm}) S(G^{\pm}) K^*(G^{\pm}) \tag{by (i) and (ii)} \\
&= K(G^{\pm}) S_{\theta}(G^{\pm}) K^*(G^{\pm}) \tag{by (iii)} \\
&= \NH_{\eta}(G). \tag{by Theorem~\ref{01}}
\end{align*}
In addition,
\begin{align*}
2L^*L - I &= 2S(G^{\pm})K^*(G^{\pm})K(G^{\pm})S(G^{\pm}) - S(G^{\pm})S(G^{\pm}) \tag{by (i)} \\
&= S(G^{\pm}) (2K^*(G^{\pm})K(G^{\pm}) - I)S(G^{\pm}) \\
&= S(G^{\pm}) C(G^{\pm}) S(G^{\pm}),
\end{align*}
so we have
\begin{align*}
U^{(w_0, \theta)} &= S^{(\theta)} (2L^*L-I) \\
&= S_{\theta}(G^{\pm}) S(G^{\pm}) C(G^{\pm}) S(G^{\pm}) \tag{by (ii)} \\
&= S(G^{\pm})S(G^{\pm}) S_{\theta}(G^{\pm}) S(G^{\pm}) C(G^{\pm}) S(G^{\pm}) \\
&= S(G^{\pm})S_{\theta}(G^{\pm})C(G^{\pm}) S(G^{\pm})  \tag{by (iii)} \\
&= S(G^{\pm}) U_{\theta} S(G^{\pm}).
\end{align*}
We have the desired statement.
\end{proof}

Theorem~\ref{21} implies that
all the spectral information of the normalized $\eta$-Hermitian adjacency matrix
corresponding to a continuous-time random walks can be reproducted.
Moreover,
we will see that from the second and third terms of RHS of (\ref{a30}) in Theorem~\ref{21},
a part of information of all the cycles in given graphs can be used for the identification of digraphs.

In \cite{HKSS2014},
the values $m_{\pm 1}$ are also determined
with a general weight function and a 1-form function.
In the quantum walks we consider,
the situations of case division can be summarized slightly simpler.

Fix a digraph $G$ with an $\eta$-function $\theta$.
For a function $f : A(G^{\pm}) \to \MB{C}$ and a closed path $c = (a_1, \dots, a_r)$,
we define
\[ \int_{c} \arg f = \sum_{j=1}^r \left\{ \arg(f(a_j)) - \arg(f(a_j^{-1})) \right\}. \]
In determining $m_{\pm 1}$,
the value
\[ \int_c \arg \tilde{w} \]
in \cite{HKSS2014} is the key.
In our situation,
the modified weight function $\tilde{w} : A(G^{\pm}) \to \MB{C}$
is nothing but the one defined by
\[ \tilde{w}(a) = \frac{1}{\sqrt{\deg o(a)}} e^{i \theta(a) / 2}, \]
so for a closed path $c = (a_1, \dots, a_r)$, we have
\begin{align*}
\int_c \arg \tilde{w}
&= \sum_{j=1}^r \left\{ \arg(\tilde{w}(a_j)) - \arg(\tilde{w}(a_j^{-1})) \right\} \\
&= \sum_{j=1}^r \left\{ \frac{\theta(a_j)}{2} - \frac{\theta(a_j^{-1})}{2} \right\} \\
&= \sum_{j=1}^r \theta(a_j).
\end{align*}
Again,
for a digraph $G$ equipped with an $\eta$-function $\theta$,
we define the function $\MC{I} : \MC{C}(G^{\pm}) \to \MB{C}$ by
\[ \MC{I}(c) = \sum_{j=1}^r \theta(a_j) \]
for a closed path $c = (a_1, \dots, a_r)$.
We consider four situations:
\begin{enumerate}[(i)]
\item $G^{\pm}$ is bipartite and $\MC{I}(c) \in 2\pi\MB{Z}$ for any closed path $c$;
\item $G^{\pm}$ is non-bipartite and $\MC{I}(c) \in 2\pi\MB{Z}$ for any closed path $c$;
\item $G^{\pm}$ is non-bipartite and for any closed path $c$,
\[ \MC{I}(c) \in
\begin{cases}
2\pi\MB{Z} \quad &\text{if $c$ is an even length closed path,} \\
2\pi(\MB{Z} + \frac12) \quad &\text{if $c$ is an odd length closed path;}
\end{cases} \]
\item otherwise.
\end{enumerate}
Then, Lemma~2 in \cite{HKSS2014} can be rewritten as follows.

\begin{pro} \label{11}
{\it
Let $G$ be a digraph equipped with an $\eta$-function $\theta$.
Denote $m_{\pm 1}$ by the multiplicities of eigenvalues $\pm 1$ of $\NH_{\eta}(G)$,
respectively.
Then we have
\[ (m_1, m_{-1}) =
\begin{cases}
(1,1) \quad &\text{case (i),} \\
(1,0) \quad &\text{case (ii),} \\
(0,1) \quad &\text{case (iii),} \\
(0,0) \quad &\text{case (iv).}
\end{cases}
\]
}
\end{pro}

In our quantum walk,
we consider the transition probability $p$ defined by
$p(e) = 1/\deg_{G^{\pm}}o(e)$,
but $p$ is reversible by Lemma~\ref{91},
so the condition on reversibility of Lemma~2 in \cite{HKSS2014} vanishes.

As a consequence of this proposition,
a necessary and sufficient conditions for the Hermitian adjacency matrix
of a $k$-regular digraph to have eigenvalues $\pm k$ is clarified.
This may be interesting also from the viewpoint of spectral graph theory.

\begin{cor}
{\it
Let $G$ be a weekly connected $k$-regular digraph
equipped with the $\frac{\pi}{2}$-function $\theta$.
Then the Hermitian adjacency matrix $H = H_{\frac{\pi}{2}}(G)$
has at most one eigenvalue $\pm k$, respectively.
In addition,
\begin{enumerate}[(1)]
\item $H$ has an eigenvalue $k$ if and only if
$\MC{I}(c) \in 2\pi\MB{Z}$ for any closed path $c$ on $G^{\pm}$.
\item $H$ has an eigenvalue $-k$ if and only if
$G$ is in the case (i) or (iii).
\end{enumerate}
}
\end{cor}

\begin{proof}
Since $G$ is $k$-regular,
we have $\NH_{\frac{\pi}{2}}(G) = \frac{1}{k}H$.
Thus, the statement follows from Proposition~\ref{11}.
\end{proof}

\section{Guo--Mohar's $H$-cospectral mates for $K_n$}

Let $\eta \in \MB{R}$ and let $G$ and $G'$ be digraphs.
$G'$ is {\it $H_{\eta}$-cospectral for $G$}
if $\Spec(H_{\eta}(G)) = \Spec(H_{\eta}(G'))$.
Guo and Mohar \cite{GM} determined all $H$-cospectral mates
for the complete graph $K_n$.
They are the following graphs.
Let $a \in \{0,1,\dots,n\}$.
The digraph $Y_{a,n-a}$ is defined by
\begin{align*}
V(Y_{a,n-a}) &= \{1,2, \dots, a\} \sqcup \{a+1, a+2, \dots, n\} = \{1,2, \dots, n\}, \\
A(Y_{a,n-a}) &= \{ (x,y) \mid x,y \in [a], x \neq y \} \\
& \qquad \sqcup \{ (x,y) \mid x,y \in [n] \setminus [a], x \neq y \} \\
& \qquad \sqcup \{ (x,y) \mid x \in [a], y \in [n] \setminus [a] \},
\end{align*}
where $[n] = \{1,2, \dots, n\}$.
For convenience,
we call the first $a$ vertices the {\it upper vertices}
and the last $n-a$ vertices the {\it lower vertices}.
Roughly speaking,
this digraph is the two complete graphs $K_a$ and $K_{n-a}$
with all possible arcs from $K_a$ to $K_{n-a}$.
Note that $Y_{n,0} = Y_{0,n} = K_n$.

\begin{pro}[Proposition~8.6 in \cite{GM}]
{\it
For each $n$,
there are precisely $n$ non-isomorphic $H$-cospectral digraphs for $K_n$.
These are the digraphs $Y_{a,n-a}$ for $a \in \{0, \dots, n-1\}$.
}
\end{pro}

This digraphs $Y_{a,n-a}$ can be obtained from $K_n$ by the following deformation.
Let $G$ be a digraph and $S \subset V(G)$.
Define $\delta(S) = \{ (x,y) \in A(G) \mid |\{x,y\} \cap S| = 1 \}$.
If $\delta(S)$ contains only digons,
then $G$ and the digraphs obtained
by replacing each digon $\{x,y\}$ with $x \not\in S$ and $y \in S$
by the arc $(x,y)$ have the same spectrum
in the sense of the Hermitian adjacency matrix (Proposition~8.3 in \cite{GM}).
This can be generalized to the $\eta$-Hermitian adjacency matrix for any $\eta \in \MB{R}$.

\begin{pro} \label{21x}
{\it
Fix $\eta \in \MB{R}$.
Let $G$ be a digraph and $S \subset V(G)$ such that $\delta(S)$ contains only digons.
Then $G$ and the digraphs obtained
by replacing each digon $\{x,y\}$ with $x \not\in S$ and $y \in S$
by the arc $(x,y)$ have the same spectrum
in the sense of the $\eta$-Hermitian adjacency matrix.
}
\end{pro}

\begin{proof}
We denote $G'$ by the digraph obtained by the above deformation.
Define the diagonal matrix $M$ indexed by $V(G)$ by
\[ M_{uu} =
\begin{cases}
e^{\eta i} \quad &\text{$u \in S$,} \\
1 \quad &\text{$u \not\in S$.}
\end{cases}
\]
Clearly, $M^{-1} = M^*$.
Since $\delta(S)$ contains only digons,
$M^* H_{\eta} M$ is again $\{0,1,e^{\pm \eta i}\}$-matrix
and this is nothing but the $\eta$-Hermitian adjacency matrix of $G'$.
Thus, $G$ and $G'$ are $H_{\eta}$-cospectral.
\end{proof}

Moreover, we can also see the following.
Recall that the function $\varphi$ is defined by
$\varphi(z) = (z+z^{-1})/2$.

\begin{lem} \label{24}
{\it
Suppose $n \geq 3$.
Let $G$ be a digraph $Y_{a,n-a}$ equipped with an $\eta$-function $\theta$
for $a \in \{0,1, \dots, n-1\}$.
Then the spectrum of $U_{\theta} = U_{\theta}(G)$ is
\[ \Spec(U_{\theta}) = \left\{
1^{(n(n-1)/2-n+2)}, -1^{(n(n-1)/2-n)},
\varphi^{-1} \left( \left\{ -\frac{1}{n-1} \right\}^{(n-1)} \right) 
\right\}. \]
In particular,
the digraphs $Y_{a,n-a}$ are $U_{\theta}$-cospectral for $K_n$.
}
\end{lem}

\begin{proof}
The $H_{\eta}$-spectrum of $K_n$ is $\{ n-1^{(1)}, -1^{(n-1)} \}$.
By Proposition~\ref{21x},
the $H_{\eta}$-spectrum of $Y_{a,n-a}$ is also $\{ n-1^{(1)}, -1^{(n-1)} \}$.
Since $Y_{a,n-a}^{\pm}$ is $(n-1)$-regular,
we have $\Spec(\NH_{\eta}(Y_{a,n-a})) = \left\{ 1^{(1)}, -\frac{1}{n-1}^{(n-1)} \right\}$.
Considering the inverse image of $\varphi$,
a set of eigenvalues
\begin{equation} \label{22}
\left\{
1^{(1)}, \varphi^{-1} \left( \left\{ -\frac{1}{n-1} \right\}^{(n-1)} \right) 
\right\}
\end{equation}
of $U_{\theta}(Y_{a,n-a})$ is inherited from $\NH_{\eta}(Y_{a,n-a})$.
In order to determine the multiplicity of the eigenvalues $\pm 1$,
we next investigate the value $\MC{I}(c)$ for any closed path $c$ on $Y_{a,n-a}^{\pm}$.
Since $c$ is a closed path,
the number of arcs from an upper vertex to a lower vertex and
the one from a lower vertex to an upper vertex have to be same.
This implies that $\MC{I}(c)$ is always zero.
Thus, $Y_{a,n-a}$ is always in the case (ii) of Proposition~\ref{11},
so we have $(m_1, m_{-1}) = (1,0)$.
Therefore, a set of eigenvalues
\begin{equation} \label{23}
\left\{
1^{(n(n-1)/2-n+1)}, -1^{(n(n-1)/2-n)}
\right\}
\end{equation}
of $U_{\theta}(Y_{a,n-a})$ is born.
Combining (\ref{22}) and (\ref{23}),
we have the statement.
\end{proof}

We summarize this section.
By Proposition~\ref{21x},
we see that the $H$-cospectral mates $Y_{a,n-a}$ for $K_n$ constructed by Guo and Mohar
are actually $H_{\eta}$-cosprctral for $K_n$.
Moreover,
they have the same spectrum in the sense of the transfer matrix $U_{\theta}$
with any $\eta$-function $\theta$ by Lemma~\ref{24}.
In the next section,
we suggest a positive support of $U_{\theta}$ in a sense.
This enables us to partially distinguish the digraphs $Y_{a,n-a}$ by their spectrum.

\section{The positive and negative supports}

As Emms, Hancock, Severini and Wilson \cite{EHSW}
attempted to distinguish strongly regular graphs
by the spectrum of matrices come from quantum walks,
we also want to define new matrices
similar to the positive support of the Grover transfer matrix
and consider the isomorphism problem of graphs 
from the viewpoint of quantum walks.

\subsection{The positive support of the Grover transfer matrix}

For a real matrix $M$,
the {\it positive support} of $M$, denoted by $M^+$,
is the $\{0,1\}$-matrix obtained from $M$ as follows:
\[
(M^+)_{xy} = \begin{cases}
1 & \quad \text{if $M_{xy} > 0$,} \\
0 & \quad \text{otherwise.}
\end{cases}
\]
Also, we define the {\it negative support} $M^{-}$ of $M$
by swapping the orientation of the inequality sign.
The following is a basic formula on the positive support of the Grover transfer matrix.

\begin{lem}[\cite{EHSW}] \label{33}
{\it
Let $G$ be an undirected connected $k$-regular graph
and $U = U(G)$ be the Grover transfer matrix.
Then we have
\[ U^+ = kSK^*K - S. \]
}
\end{lem}

Consideration of the positive support of the $n$-th power
has been done in various papers.
Godsil and Guo \cite{GG} revealed a clear formula
on the positive support $(U^2)^+$ of the square of $U$,
which is that
\[ (U^2)^+ = (U^+)^2 + I \]
holds for $k$-regular graphs with $k \geq 3$.
In addition,
consideration of $(U^3)^+$ was carried out in~\cite{Guo,HKSS2013}.
More generally, $(U^n)^+$ has been investigated in~\cite{KSS2019}.
On the other hand, 
those consideration always has assumptions on regularity or a condition on girth,
and it seems that it has not been investigated much in general situation so far.

\subsection{Our transfer matrix}

Let $G$ be a digraph equipped with an $\eta$-function $\theta$.
Define the diagonal matrix $D_{\theta}$ indexed by $A(G^{\pm})$, where
\[ (D_{\theta})_{ab} = e^{\theta(a) i} \delta_{ab}. \]
Now, we want to define the positive support of
our transfer matrix $U_{\theta} = U_{\theta}(G)$,
but this is a complex matrix, so we have to consider how to define its positive support.
Considering the similarity to the one on undirected graphs,
we propose the following definition.

\begin{defi} \label{31}
Let $G$ be a digraph equipped with an $\eta$-function $\theta$.
For a positive integer $n$,
we define {\it the positive support of the $n$-th power of the transfer matrix},
denoted by $U_{\theta}^{(n,+)} = U_{\theta}(G)^{(n,+)}$,
which is indexed by $A(G^{\pm})$, where
\[
(U_{\theta}^{(n,+)})_{ab} = \begin{cases}
1 & \quad \text{if $\re(D_{\theta}U_{\theta}^n)_{ab} > 0$,} \\
0 & \quad \text{otherwise.}
\end{cases}
\]
\end{defi}

We will also consider the negative support of $U_{\theta}$.
Define $U_{\theta}^{(n,-)}$ by swapping the orientation of the inequality sign.

If $G$ is an undirected graph,
then $\theta$ is the zero function,
so $D_{\theta} = I$ and $U_{\theta}(G) = U(G)$,
where $U(G)$ is the Grover transfer matrix of the undirected graph $G$.
Then $D_{\theta}U_{\theta}^n = U(G)^n$,
so we have $U_{\theta}(G)^{(n,+)} = (U(G)^n)^+$ for any positive integer $n$.

In the discussion below,
our quantum walk on a digraph $G$ and the Grover walk on $G^{\pm}$ are intermixed,
so we need to read cautiously.

\begin{lem} \label{32}
{\it
Let $G$ be a digraph equipped with an $\eta$-function $\theta$.
Then we have
\begin{enumerate}[(i)]
\item $D_{\theta}S_{\theta} = S(G^{\pm})$;
\item $D_{\theta}U_{\theta} = U(G^{\pm})$.
\end{enumerate}
}
\end{lem}

\begin{proof}
\begin{enumerate}[(i)]
\item This statement can be checked by calculating the components of both sides.
\item Remarking that $C = C(G) = C(G^{\pm})$, we have
\begin{align*}
D_{\theta} U_{\theta} &= D_{\theta} S_{\theta} C \\
&= S(G^{\pm}) C \tag{by (i)} \\
&= U(G^{\pm}).
\end{align*}
\end{enumerate}
\end{proof}

By this lemma, we see that the positive support of the first power of our transfer matrix
is the same as the one of the Grover transfer matrix as follows.
The reason for giving Definition~\ref{31} is to make this equality hold.

\begin{pro} \label{61}
{\it
Let $G$ be a digraph equipped with an $\eta$-function $\theta$.
For $\varepsilon \in \{+,-\}$, we have
\[ U_{\theta}^{(1,\varepsilon)} = U^{\varepsilon}, \]
where $U = U(G^{\pm})$.
If $G$ is $k$-regular,
\[ U_{\theta}^{(1,+)} = kSK^*K - S \]
holds, where $S = S(G^{\pm})$.
}
\end{pro}

\begin{proof}
By (ii) of Lemma~\ref{32}, we have 
\[ \re (D_{\theta}U_{\theta})_{ab} > 0 \iff U(G^{\pm})_{ab} > 0 \]
for any arcs $a,b \in A(G^{\pm})$.
Considering the negative support as well, we see that the first statement follows.
If $G$ is $k$-regular,
then $G^{\pm}$ is also $k$-regular,
so we have the second statement by Lemma~\ref{33}.
\end{proof}

\subsection{The negative support of the Grover transfer matrix}

Before discussing the positive support $U_{\theta}^{(2,+)}$ of the square,
we have to consider the negative support of $U$ for undirected graphs.
Actually, we will see that $U_{\theta}^{(2,\varepsilon)}$ can be expressed
by not only the positive also negative supports of $U^2$.


Let $G$ be a $k$-regular undirected graph and $U = U(G)$.
Recalling that
\[ U_{ab} = \frac{2}{k} \delta_{t(b), o(a)} - \delta_{a^{-1}, b}, \]
we see that $U_{ab} < 0$ if and only if $\delta_{a^{-1}, b} = 1$.
Thus, we have the following.

\begin{pro}
{\it
Let $G$ be a $k$-regular undirected graph with $k \geq 3$ and
$U$ be the Grover transfer matrix of $G$.
Then we have
\[ U^- = S. \]
}
\end{pro}

Next, we discuss the negative support $(U^2)^-$ of the square.
For general description, digraphs are assumed.
Let $G = (V,A)$ be a $k$-regular digraph equipped with an $\eta$-function $\theta$
and $U = U(G^{\pm})$ be the Grover transfer matrix of $G^{\pm}$.
For arcs $a,b \in A^{\pm 1}$,
define the set $\MC{A}(a,b)$ as follows:
\[ \MC{A}(a,b) =
\left\{
z \in A^{\pm 1} \mid e^{-\theta(z)i}
U_{az}U_{zb}
\neq 0
\right\}. \]

Classification of arcs $a,b \in A^{\pm}$ to give $|\MC{A}(a,b)| = 1$ is the key to
determine the positive and negative supports of the square.

\begin{lem} \label{73}
{\it
Suppose $k \geq 3$.
Let $G$ be a $k$-regular digraph equipped with an $\eta$-function $\theta$.
For arcs $a,b \in A(G)^{\pm 1}$, we have $|\MC{A}(a,b)| \leq 1$,
and the equality holds if and only if either of the following happens.
\begin{enumerate}[(i)]
\item $a = b$;
\item $o(a) = o(b)$ and $a \neq b$; 
\item $t(a) = t(b)$ and $a \neq b$; 
\item $t(b) \sim o(a)$ in $G^{\pm}$, but
the arcs $a,b$ are neither in (i), (ii) nor (iii).
\end{enumerate}
} 
\end{lem}

\begin{proof}
Suppose there exists an arc $z \in \MC{A}(a,b)$.
Then we have
\begin{equation} \label{71}
\frac{2}{k} \delta_{t(z),o(a)} - \delta_{a,z^{-1}} \neq 0
\end{equation}
and
\begin{equation} \label{200}
\frac{2}{k} \delta_{t(b),o(z)} - \delta_{z,b^{-1}} \neq 0.
\end{equation}
Remarking that $(\delta_{t(z), o(a)}, \delta_{a, z^{-1}}) = (0,1)$ does not happen,
Equality~(\ref{71}) is equivalent to $(\delta_{t(z), o(a)}, \delta_{a, z^{-1}})$ $= (1,0), (1,1)$.
Dealing with (\ref{200}) similarly,
there are precisely $2 \times 2$ cases which we have to consider.
In each case, we have $\delta_{t(z), o(a)} = \delta_{t(b), o(z)} = 1$,
so $z$ is determined to be the arc $(t(b), o(a))$.
This implies that $|\MC{A}(a,b)| \leq 1$.
Moreover, the $2 \times 2$ cases are nothing but the ones of
(i), (ii), (iii), (iv) in our statement.
\end{proof}

\begin{cor} \label{82}
{\it
Let $G$ be a $k$-regular undirected graph and $U = U(G)$ be the Grover transfer matrix.
For arcs $a,b \in A(G)$,
$(U^2)_{ab} > 0$ if and only if $(a,b)$ is in ether (i) or (iv) of Lemma~\ref{73},
and $(U^2)_{ab} < 0$ if and only if $(a,b)$ is in ether (ii) or (iii) of Lemma~\ref{73}.
}
\end{cor}

\begin{proof}
By Lemma~\ref{73},
we see that $(U^2)_{ab} \neq 0$ if and only if $|\MC{A}(a,b)| = 1$.
From this, we have $(U^2)_{ab} > 0$ if $(a,b)$ is in ether (i) or (iv),
and $(U^2)_{ab} < 0$ if $(a,b)$ is in ether (ii) or (iii).
\end{proof}

By this corollary,
we are interested in the cases (ii) and (iii) to find $(U^2)^-$.
In order to describe these situations in matrix,
we introduce the following matrices.

Let $G$ be a digraph.
We define the two matrices $F_t$ and $F_o$ whose rows are indexed by $V(G)$ and
columns are indexed by $A(G)^{\pm 1}$, respectively, where
\[ (F_t)_{x,a} = \delta_{x,t(a)} \]
and
\[ (F_o)_{x,a} = \delta_{x,o(a)}. \]


\begin{lem} \label{54}
{\it
Let $G$ be a digraph equipped with an $\eta$-function $\theta$.
Then we have
\begin{enumerate}[(i)]
\item $(F_o^{\top} F_t)_{ab} = \delta_{t(b), o(a)}$;
\item $SF_t^{\top} = F_{o}^{\top}$ and $SF_o^{\top} = F_{t}^{\top}$,
\end{enumerate}
where $S = S(G^{\pm})$.
}
\end{lem}

\begin{proof}
Proven by direct calculation.
\end{proof}

On regular digraphs, we have the following.

\begin{lem} \label{76}
{\it
Suppose $k \geq 3$.
Let $G$ be a $k$-regular digraph equipped with an $\eta$-function $\theta$.
Then we have
\begin{enumerate}[(i)]
\item $U_{\theta} = D_{\theta}^{-1} (\frac{2}{k}F_o^{\top}F_t - S)$;
\item $U_{\theta}^{(1,+)} = F_o^{\top}F_t - S$,
\end{enumerate}
where $S = S(G^{\pm})$.
}
\end{lem}

\begin{proof}
In order to prove (i),
we first state
\begin{equation} \label{52}
U(G^{\pm}) = \frac{2}{k}F_o^{\top}F_t - S.
\end{equation}
This is proven by
\begin{align*}
U(G^{\pm})_{ab} &= \frac{2}{k}\delta_{t(b), o(a)} - \delta_{a, b^{-1}} \tag{by (\ref{53})} \\
&= \frac{2}{k}(F_o^{\top} F_t)_{ab} - S_{ab} \tag{by (i) of Lemma~\ref{54}} \\
&= \left( \frac{2}{k}F_o^{\top}F_t - S \right)_{ab}.
\end{align*}
Therefore, we have
\begin{align*}
U_{\theta} &= D_{\theta}^{-1} U(G^{\pm}) \tag{by (ii) of Lemma~\ref{32}} \\
&= D_{\theta}^{-1} \left( \frac{2}{k}F_o^{\top}F_t - S \right). \tag{by (\ref{52})}
\end{align*}

We next prove (ii).
Remark that $U_{\theta}^{(1,+)} = U(G^{\pm})^+$ by Proposition~\ref{61}.
From (\ref{53}) and $k \geq 3$,
we have
\begin{align*}
(U_{\theta}^{(1,+)})_{ab} &= 
  \begin{cases}
1 \quad &\text{if $o(a) = t(b)$ and $a \neq b^{-1}$,} \\
0 \quad &\text{otherwise.}
  \end{cases} \\
&= \delta_{o(a), t(b)}(1 - \delta_{a,b^{-1}}) \\
&= \delta_{o(a), t(b)} - \delta_{o(a), t(b)}\delta_{a,b^{-1}} \\
&= \delta_{o(a), t(b)} - \delta_{a,b^{-1}} \\
&= (F_o^{\top}F_t - S)_{ab} \tag{by (i) of Lemma~\ref{54}}.
\end{align*}
\end{proof}

Now, we give the structure of $(U^2)^-$.

\begin{pro}
{\it
Suppose $k \geq 3$.
Let $G$ be a $k$-regular undirected graph and
$U$ be the Grover transfer matrix of $G$.
Then we have
\[ (U^2)^- = S U^+ + U^+ S. \]
}
\end{pro}

\begin{proof}
For $\varepsilon \in \{o,t\}$,
we can confirm that
\[
(F_{\varepsilon}^{\top}F_{\varepsilon} - I)_{ab} =
\delta_{\varepsilon(a), \varepsilon(b)} - \delta_{a,b} =
\begin{cases}
1 \quad &\text{if $\varepsilon(a) = \varepsilon(b)$ and $a \neq b$,} \\
0 \quad &\text{otherwise,}
\end{cases} \]
so by Corollary~\ref{82}, we have
\begin{align*}
(U^2)^- &= F_t^{\top}F_t - I + F_o^{\top}F_o - I \\
&= SF_o^{\top}F_t  + F_o^{\top}F_t S - 2I \tag{by (ii) of Lemma~\ref{54}} \\
&= S(U^+ + S) + (U^+ + S)S - 2I \tag{by (ii) of Lemma~\ref{76}} \\
&= SU^+ + U^+S.
\end{align*}
\end{proof}

\subsection{The positive and negative supports of the square of our transfer matrix}

In this subsection,
we discuss the positive and negative supports of the square of transfer matrices.
Let $\varepsilon \in \{+,-\}$.
In the case of $U_{\theta}^{(2,\varepsilon)}$,
the argument must be divided by the value of $\eta$.

Let $G = (V, A)$ be a digraph.
Then the digraph $G^{-1} = (V,A^{-1})$ is so-called the {\it transpose graph} of $G$.
Also, when $G$ is equipped with an $\eta$-function $\theta$,
we sometimes write as $\theta = \theta_G$.
For an $\eta$-function $\theta$,
we define a function $-\theta$ by $(-\theta)(a) = - \theta (a)$ for any arc $a \in A(G^{\pm})$.
Clearly,
\begin{equation} \label{81}
-\theta_{G} = \theta_{G^{-1}}
\end{equation}
follows.
First, we state the following.

\begin{pro} \label{102}
{\it
Let $G$ be a digraph equipped with an $\eta$-function $\theta$.
Then we have
\[ U_{\theta}^{(2,\varepsilon)}(G) = U_{\theta}^{(2,\varepsilon)}(G^{-1}) \]
for $\varepsilon \in \{+,-\}$.
}
\end{pro}

\begin{proof}
It is enough to show that 
\[ \re(D_{\theta}(G)U_{\theta}(G)^2)_{ab} = \re(D_{\theta}(G^{-1})U_{\theta}(G^{-1})^2)_{ab} \]
holds for any arcs $a,b \in A(G^{\pm})$.
Writing as $U = U(G^{\pm})$, we can calculate as follows:
\begin{align*}
\re(D_{\theta}(G)U_{\theta}(G)^2)_{ab} 
&= \re(U D_{\theta}(G)^{-1} U)_{ab} \tag{by (ii) of Lemma~\ref{32}} \\
&= \sum_{z \in A(G^{\pm})} \re \left( e^{-\theta_{G}(z)} U_{az} U_{zb} \right) \\
&= \sum_{z \in A(G^{\pm})} \re \left( e^{-(-\theta_{G}(z))} U_{az} U_{zb} \right) \\
&= \sum_{z \in A(G^{\pm})} \re \left( e^{-(\theta_{G^{-1}}(z))} U_{az} U_{zb} \right) \tag{by (\ref{81})} \\
&= \re(U D_{\theta}(G^{-1})^{-1} U)_{ab} \\
&= \re(D_{\theta}(G^{-1})U_{\theta}(G^{-1})^2)_{ab}.
\end{align*}
\end{proof}

From the above proposition,
we unfortunately cannot distinguish $G$ and $G^{-1}$ by $U_{\theta}^{(2,\varepsilon)}$.

\begin{lem} \label{41}
{\it
Let $G$ be a digraph equipped with an $\eta$-function $\theta$.
Then we have
\[ U_{-\theta}(G) = U_{\theta}(G^{-1}). \]
}
\end{lem}

\begin{proof}
By (\ref{81}), we have
\begin{align*}
(S_{-\theta}(G))_{ab} &= e^{(-\theta_{G})(b) i} \delta_{a,b^{-1}} \\
&= e^{\theta_{G^{-1}}(b) i} \delta_{a,b^{-1}} \\
&= (S_{\theta}(G^{-1}))_{ab},
\end{align*}
so $U_{-\theta}(G) = S_{-\theta}(G)C = S_{\theta}(G^{-1})C = U_{\theta}(G^{-1})$.
\end{proof}

For an $\eta$-function $\theta$,
the function $-\theta$ is nothing but the $(-\eta)$-function.
By Lemma~\ref{41},
we need to consider only the range $0 \leq \eta \leq \pi$
since we consider $\eta$ to be the rotation angle.

Now, we determine the structure of $U_{\theta}^{(2,\varepsilon)}$
for $\varepsilon \in \{+,-\}$.
For two arcs $a,b$,
we define the ordered pair $m(a,b) = (t(b), o(a))$.
Remark that $m(a,b)$ does not necessarily belong to the arc set.
Let $G$ be a $k$-regular digraph equipped with an $\eta$-function $\theta$ and $U = U(G^{\pm})$.
Observe that
\[ (D_{\theta}U_{\theta}^2)_{ab} = \sum_{z \in A(G^{\pm})} e^{-\theta(z)i} U_{az} U_{zb} \]
and by Lemma~\ref{73},
if the contents of the sum appear, it is just one and it is nothing but $z = m(a,b)$.
Then
\begin{equation} \label{92}
(D_{\theta}U_{\theta}^2)_{ab} = e^{-\theta(m(a,b))i} (U^2)_{ab},
\end{equation}
so the value of $(D_{\theta}U_{\theta}^2)_{ab}$ is,
roughly speaking, either rotated $U^2_{ab}$
or non-rotated $U^2_{ab}$,
and it depends on whether $m(a,b) \in A(G) \cap A(G)^{-1}$ or not.
Therefore, we define the following matrix $R$,
which is indexed by $A(G^{\pm})$, such that
\[ R_{ab} =
\begin{cases}
1 \quad &\text{if $m(a,b) \in A(G) \cap A(G)^{-1}$,} \\
0 \quad &\text{otherwise.}
\end{cases}
\]
By using this matrix,
$U_{\theta}^{(2,\varepsilon)}$ can be described depending on the value of $\eta$ as follows.
For $\varepsilon \in \{+,-\}$,
we denote the element of $\{+,-\} \setminus \{\varepsilon\}$ by $-\varepsilon$.

\begin{thm} \label{100}
{\it
Let $G$ be a digraph equipped with an $\eta$-function $\theta$
and $U = U(G^{\pm})$ be the Grover transfer matrix of $G^{\pm}$.
For $\varepsilon \in \{+,-\}$, we have
\[ U_{\theta}^{(2,\varepsilon)} =
\begin{cases}
(U^2)^{\varepsilon} \quad &\text{if $0 \leq \eta < \frac{\pi}{2}$,} \\
(U^2)^{\varepsilon} \circ R \quad &\text{if $\eta = \frac{\pi}{2}$,} \\
(U^2)^{\varepsilon} \circ R + (U^2)^{-\varepsilon} \circ (J-R) &\text{if $\frac{\pi}{2} < \eta \leq \pi$,}
\end{cases}
\]
where $A \circ B$ is the Hadamard product of $A$ and $B$,
which is defined by $(A \circ B)_{xy} = A_{xy}B_{xy}$.
}
\end{thm}

\begin{proof}
We only provide a proof in the case of $\varepsilon = +$.
Fix arcs $a,b \in A(G^{\pm})$.
If $m(a,b) \not\in A(G^{\pm})$, then $(D_{\theta}U_{\theta}^2)_{ab} = 0$ by Lemma~\ref{73}.
From (\ref{92}),
\begin{equation} \label{201}
(D_{\theta}U_{\theta}^2)_{ab} =
\begin{cases}
(U^2)_{ab} \qquad &\text{if $m(a,b) \in A(G) \cap A(G^{-1})$,} \\
e^{\mp i \eta} (U^2)_{ab} \qquad &\text{if $m(a,b) \in A(G^{\pm}) \setminus (A(G) \cap A(G^{-1}))$,} \\
0 \qquad &\text{otherwise.}
\end{cases}
\end{equation}

Suppose $0 \leq \eta < \frac{\pi}{2}$.
By ($\ref{201}$),
we have $\re(D_{\theta}U_{\theta}^2)_{ab} = \re(U^2)_{ab}$,
so $U_{\theta}^{(2,+)} = (U^2)^+$ holds.

In the case of $\eta = \frac{\pi}{2}$,
we see that
$\re(D_{\theta}U_{\theta}^2)_{ab} > 0$ if and only if
$(U^2)_{ab} > 0$ and $m(a,b) \in A(G) \cap A(G)^{-1}$ by (\ref{201}).
This implies $(U_{\theta}^{(2,+)})_{ab} = ((U^2)^+)_{ab} R_{ab}$.

Similarly, in the case of $\frac{\pi}{2} < \eta \leq \pi$,
$\re(D_{\theta}U_{\theta}^2)_{ab} > 0$ if and only if
\[ (U^2)_{ab} > 0 \text{ and } m(a,b) \in A(G) \cap A(G)^{-1},\]
or
\[ (U^2)_{ab} < 0 \text{ and } m(a,b) \in A(G^{\pm}) \setminus (A(G) \cap A(G^{-1})) \]
by (\ref{201}).
This implies $(U_{\theta}^{(2,+)})_{ab} = ((U^2)^+)_{ab} R_{ab} + ((U^2)^-)_{ab} (1-R_{ab})$.
Summarizing the above, we have the statement.
\end{proof}

By this theorem,
we see that the structure of $U_{\theta}^{(2,\varepsilon)}(G)$
is determined by the underlying graph $G^{\pm}$ if $0 \leq \eta < \frac{\pi}{2}$.
On the other hand, if $\frac{\pi}{2} \leq \eta \leq \pi$,
the information of digraphs lives.
However, if a digraph $G$ does not have any digons,
$R$ is determined to be the zero matrix,
so we have
\[ U_{\theta}^{(2,\varepsilon)} =
\begin{cases}
(U^2)^{\varepsilon} \quad &\text{if $0 \leq \eta < \frac{\pi}{2}$,} \\
O \quad &\text{if $\eta = \frac{\pi}{2}$,} \\
(U^2)^{-\varepsilon} &\text{if $\frac{\pi}{2} < \eta \leq \pi$.}
\end{cases}
\]
In any case, it may be interesting that
the structure of $U_{\theta}^{(2,\varepsilon)}$ is represented
by the positive and negative supports of the underlying graph
and it moreover changes by the value of $\eta$.

\section{Counting digons}

In this section,
we find the number of digons via $U_{\theta}^{(2,+)}$.
As a consequence, we show that
Guo--Mohar's digraphs $Y_{a,n-a}$ which could not be identified
by eigenvalues of $H, H_{\eta}, U_{\theta}, U_{\theta}^{(1,\varepsilon)}$
can be half identified by $U_{\theta}^{(2,+)}$.

\begin{pro} \label{101}
{\it
Let $G$ be a digraph equipped with an $\eta$-function $\theta$.
We denote the number of digons in $G$ by $d$.
Then we have
\[
\Tr(U_{\theta}^{(2,+)}) =
\begin{cases}
2 |E(G^{\pm})| \quad &\text{if $0 \leq \eta < \frac{\pi}{2}$,} \\
2d \quad &\text{if $\frac{\pi}{2} \leq \eta \leq \pi$.}
\end{cases}
\]
}
\end{pro}

\begin{proof}
Let $U = U(G^{\pm})$.
For any arc $a \in A(G^{\pm})$,
the pair of arcs $(a,a)$ is in (i) of Lemma~\ref{73},
so $((U^2)^+)_{aa} = 1$ by Corollary~\ref{82}.
Suppose $0 \leq \eta < \frac{\pi}{2}$.
By Theorem~\ref{100}, we have
\[ \Tr(U_{\theta}^{(2,+)}) = \Tr((U^2)^+)
= \sum_{a \in A(G^{\pm})} ((U^2)^+)_{aa} = |A(G^{\pm})| = 2|E(G^{\pm})|. \]
Suppose $\frac{\pi}{2} \leq \eta \leq \pi$.
Since $R_{aa} = 1$ if and only if $a \in A(G) \cap A(G)^{-1}$,
\begin{align*}
\Tr(U_{\theta}^{(2,+)}) &= \Tr((U^2)^+ \circ R) \tag{by Theorem~\ref{100}} \\
&= \sum_{a \in A(G^{\pm})} ((U^2)^+)_{aa} R_{aa} \\
&= \sum_{a \in A(G) \cap A(G)^{-1}} 1 \\
&= 2d.
\end{align*}
\end{proof}

Using this, we see that Guo--Mohar's digraphs $Y_{a,n-a}$
can be half identified by eigenvalues of $U_{\theta}^{(2,+)}$.

\begin{cor}
{\it
Let $\frac{\pi}{2} \leq \eta \leq \pi$ and we consider the $\eta$-function $\theta$.
Suppose $a \geq n-a$ and $b \geq n-b$,
i.e., $a,b \geq \frac{n}{2}$.
If $a \neq b$, then
\[ \Phi(U_{\theta}^{(2,+)}(Y_{a, n-a})) \neq \Phi(U_{\theta}^{(2,+)}(Y_{b, n-b})), \]
where $\Phi(M)$ denotes the characteristic polynomial of a matrix $M$.
}
\end{cor}

\begin{proof}
We denote the number of digons in the digraph $Y_{k, n-k}$ by $d(k)$.
Then 
\[ d(k) = \frac{k(k-1)}{2} + \frac{(n-k)(n-k-1)}{2}. \]
Without loss of generality,
we can assume $a>b$.
Since $a+b > n$, we have
\[ d(a) - d(b) = (a-b)(a+b-n) > 0. \]
In particular, $d(a) \neq d(b)$.
By (ii) of Proposition~\ref{101}, 
it can be seen that the coefficients of the characteristic polynomials are different.
\end{proof}

Note that we have $d(a) = d(n-a)$, but
$Y_{a, n-a}$ and $Y_{n-a, a}$ are not isomorphic to each other unless $a \in \{0,n\}$.
However, $Y_{n-a, a}^{-1} = Y_{a, n-a}$,
so we unfortunately have
$\Phi(U_{\theta}^{(2,+)}(Y_{a, n-a})) = \Phi(U_{\theta}^{(2,+)}(Y_{n-a, a}))$
by Proposition~\ref{102}.
The two digraphs $Y_{a, n-a}$ and $Y_{n-a, a}$ cannot be identified
by the spectrum of $U_{\theta}^{(2,+)}$.

\section{Tables}

Up to the previous section,
we have obtained several new matrices defined by digraphs.
In this section, using computer,
we observe the behavior that small digraphs can be identified
by eigenvalues of these matrices with how much.
Here, we made the following tables similar to the one for adjacency matrices
and Hermitian adjacency matrices made by Guo and Mohar \cite{GM}.

\begin{table}[H]
  \begin{tabular}{l|r|r|r|r|r}
\hline
Order&$2$&$3$&$4$&$5$&$6$\\
\hline
\hline
Number of digraphs&
$3$&
$16$&
$218$&
$9608$&
$1540944$\\
\hline
Number of distinct characteristic polynomials&
$2$&
$7$&
$46$&
$718$&
$35237$\\
\hline
Maximum size of a $A$-cospectral class&
$2$&
$6$&
$42$&
$592$&
$15842$\\
\hline
Number of digraphs determined by $A$-spectrum&
$1$&
$5$&
$23$&
$166$&
$2314$\\
\hline
Number of $A$-cospectral classes containing:&&&&&\\
 a) no graphs&
$0$&
$3$&
$35$&
$685$&
$35086$\\
b) only graphs&
$ 1$&
$ 2$&
$ 5$&
$ 15$&
$ 69$\\
c) at least one graph and a digraph&
$1$&
$ 2$&
$6$&
$18$&
$82$\\
\hline
  \end{tabular}
  \caption{The adjacency matrix spectra of small digraphs} \label{1004}
\end{table}

\begin{table}[H]
  \begin{tabular}{l|r|r|r|r|r}
\hline
Order 
&$2$&$3$&$4$&$5$&$6$\\
\hline
\hline
Number of digraphs&
$3$&
$16$&
$218$&
$9608$&
$1540944$\\
\hline
Number of distinct characteristic polynomials&
$2$&
$7$&
$41$&
$765$&
$81175$\\
\hline
Maximum size of a $H_\eta$-cospectral class&
$2$&
$6$&
$18$&
$84$&
$888$\\
\hline
Number of digraphs determined by $H_\eta$-spectrum&
$1$&
$3$&
$9$&
$82$&
$1559$\\
\hline
Number of $H_\eta$-cospectral classes containing:&&&&&\\
 a) no graphs&
$0$&
$3$&
$30$&
$732$&
$81024$\\
b) only graphs&
$ 1$&
$ 1$&
$ 1$&
$ 1$&
$ 1$\\
c) at least one graph and a digraph&
$1$&
$ 3$&
$10$&
$ 32$&
$ 150$\\
\hline
  \end{tabular}
   \caption{The $H_{\eta}$-spectra of small digraphs with $\eta = \frac{\pi}{3}$} \label{1001}
\end{table}

\begin{table}[H]
  \begin{tabular}{l|r|r|r|r|r}
\hline
Order&$2$&$3$&$4$&$5$&$6$\\
\hline
\hline
Number of digraphs&
$3$&
$16$&
$218$&
$9608$&
$1540944$\\
\hline
Number of distinct characteristic polynomials&
$2$&
$6$&
$27$&
$275$&
$10920$\\
\hline
Maximum size of a $H$-cospectral class&
$2$&
$6$&
$21$&
$158$&
$1338$\\
\hline
Number of digraphs determined by $H$-spectrum&
$1$&
$2$&
$3$&
$5$&
$16$\\
\hline
Number of $H$-cospectral classes containing:&&&&&\\
 a) no graphs&
$0$&
$2$&
$16$&
$242$&
$10769$\\
b) only graphs&
$ 1$&
$ 1$&
$ 1$&
$ 1$&
$ 1$\\
c) at least one graph and a digraph&
$1$&
$ 3$&
$10$&
$ 32$&
$ 150$\\
\hline
  \end{tabular}
   \caption{The $H$-spectra of small digraphs} \label{1003}
\end{table}

\begin{table}[H]
  \begin{tabular}{l|r|r|r|r|r}
\hline
Order 
&$2$&$3$&$4$&$5$&$6$\\
\hline
\hline
Number of digraphs&
$3$&
$16$&
$218$&
$9608$&
$1540944$\\
\hline
Number of distinct characteristic polynomials&
$2$&
$5$&
$20$&
$150$&
$3698$\\
\hline
Maximum size of a $H_\eta$-cospectral class&
$2$&
$6$&
$27$&
$243$&
$2430$\\
\hline
Number of digraphs determined by $H_\eta$-spectrum&
$1$&
$1$&
$1$&
$1$&
$1$\\
\hline
Number of $H_\eta$-cospectral classes containing:&&&&&\\
 a) no graphs&
$0$&
$1$&
$9$&
$117$&
$3547$\\
b) only graphs&
$ 1$&
$ 1$&
$ 1$&
$ 1$&
$ 1$\\
c) at least one graph and a digraph&
$1$&
$ 3$&
$10$&
$ 32$&
$ 150$\\
\hline
  \end{tabular}
     \caption{The $H_{\eta}$-spectra of small digraphs with $\eta = \frac{2}{3} \pi$} \label{1002}
\end{table}

Let $E_n = (\{ 1,2, \dots, n \}, \emptyset)$ be the empty graph with $n$ vertices.
We give tables on $U_{\theta}^{(2,+)}$,
but the transfer matrix cannot be defined from $E_n$,
so this graph is excluded from the following tables.

\begin{table}[H]
  \begin{tabular}{l|r|r|r|r|r}
\hline
Order 
&$2$&$3$&$4$&$5$&$6$\\
\hline
\hline
Number of digraphs&
$3$&
$16$&
$218$&
$9608$&
$1540944$\\
\hline
Number of distinct characteristic polynomials&
$2$&
$6$&
$34$&
$371$&
$11748$\\
\hline
Maximum size of a $U_{\pi/2}^{(2,+)}$-cospectral class&
$1$&
$6$&
$53$&
$700$&
$37013$
\\
\hline
Number of digraphs determined by $U_{\pi/2}^{(2,+)}$-spectrum&
$2$&
$4$&
$13$&
$50$&
$284$
\\
\hline
Number of $U_{\pi/2}^{(2,+)}$-cospectral classes containing:&&&&&\\
 a) no graphs&
$1$&
$3$&
$25$&
$339$&
$11598$
\\
b) only graphs&
$1$&
$3$&
$9$&
$32$&
$150$
\\
c) at least one graph and a digraph&
$0$&
$0$&
$0$&
$0$&
$0$
\\
\hline
  \end{tabular}
  \caption{The $U_{\theta}^{(2,+)}$-spectra of small digraphs except $E_n$ with $\eta = \frac{\pi}{2}$} \label{1005}
\end{table}

\begin{table}[H]
  \begin{tabular}{l|r|r|r|r|r}
\hline
Order 
&$2$&$3$&$4$&$5$&$6$\\
\hline
\hline
Number of digraphs&
$3$&
$16$&
$218$&
$9608$&
$1540944$\\
\hline
Number of distinct characteristic polynomials&
$2$&
$6$&
$45$&
$601$&
$20306$\\
\hline
Maximum size of a $U_{\eta}^{(2,+)}$-cospectral class&
$1$&
$6$&
$22$&
$204$&
$5120$
\\
\hline
Number of digraphs determined by $U_{\eta}^{(2,+)}$-spectrum&
$2$&
$4$&
$13$&
$47$&
$280$
\\
\hline
Number of $U_{\eta}^{(2,+)}$-cospectral classes containing:&&&&&\\
 a) no graphs&
$1$&
$3$&
$36$&
$569$&
$20156$
\\
b) only graphs&
$1$&
$3$&
$9$&
$27$&
$135$
\\
c) at least one graph and a digraph&
$0$&
$0$&
$0$&
$5$&
$15$
\\
\hline
  \end{tabular}
  \caption{The $U_{\theta}^{(2,+)}$-spectra of small digraphs except $E_n$
  with $\frac{\pi}{2} < \eta \leq \pi $} \label{1006}
\end{table}

\subsection{Discussion on the tables}

Characterizing graphs by eigenvalues is one of the interesting research topics in spectral graph theory.
In the case of undirected graphs,
there are many studies that characterize graphs by eigenvalues,
while there are many studies that construct {\it cospectral graphs},
which are non-isomorphic graphs with the same eigenvalues.
For example, \cite{BH, GMc, K} can be cited as references.
Chapter~14 of the first one is on spectral characterization and the other two are on constructing cospectral graphs.

In fact,
it is extremely difficult to find a matrix to characterize all graphs by eigenvalues.
Digraphs are even more difficult.
First, we compare Table~\ref{1004} and Table~\ref{1003}.
Table~\ref{1004} is obtained using the adjacency matrix,
and Table~\ref{1003} is obtained using the Hermitian adjacency matrix.
There are 1540944 digraphs with 6 vertices,
from which only 35237 distinct characteristic polynomials can be obtained.
In Table~\ref{1003}, there are 10920 distinct characteristic polynomials.
On the other hand,
according to the maximum size of a cospectral class (on 6 vertices) in both tables,
Table~\ref{1004} has 15842 and Table~\ref{1003} has 1338.
This shows that Hermitian adjacency matrices more easily find a specific digraph
from digraphs with the same eigenvalues.
In this way, the size of a cospectral class is also an important index for evaluating the goodness of matrices.
Such as $H_{\eta}$ and $U_{\theta}^{(n, +)}$,
we defined many matrices induced by digraphs. 
Table~\ref{1001},~\ref{1002},~\ref{1005},~\ref{1006} are the tables obtained from
$H_{\frac{\pi}{3}}$, $H_{\frac{2}{3}\pi}$, $U_{\theta}^{(2, +)}$ with $\eta = \frac{\pi}{2}$,
and $U_{\theta}^{(2, +)}$ with $\frac{\pi}{2} < \eta \leq \pi$, respectively.
Among these,
$H_{\frac{\pi}{3}}$ of Table~\ref{1001} is relatively good,
but $U_{\theta}^{(2, +)}$ is unexpectedly bad for any $\eta$.
It is worth investigating in the future how $\eta$ plays a role in distinguishing digraphs.
Note that the information of digraphs completely disappears
when $0 \leq \eta < \frac{\pi}{2}$ by Theorem~\ref{100}.
However, as for the transfer matrix,
there is still room for further consideration
because $U_{\theta}^{(n, +)}$ can be considered for any $n \in \MB{N}$.


\section{Summaries and remarks}

In this paper,
we proposed a quantum walk defined by digraphs,
and stated that the transfer matrix relates to Hermitian adjacency matrices.
Our quantum walk is a generalization of the Grover walk,
and there is an arbitrariness of perturbation.
If $\eta = 0$, this is nothing but the Grover walk on the underlying graph.
In Section 6,
we defined the positive and negative supports of the transfer matrix to the $n$-th power,
and clarified the structure for the case of $n = 2$.
Theorem~\ref{100} states that the structure of $U_{\theta}^{(2, \varepsilon)}$
discretely change depending on $\eta$.
However,
the values in Table~\ref{1005} and Table~\ref{1006} are by no means good.
Considerations for $n \geq 3$ should be done as our future tasks.
Finding an explicit formula on $U_{\theta}^{(n,\varepsilon)}$ is also interesting.

On the other hand,
the eigenvalues of our transfer matrix are roughly determined from
those of Hermitian adjacency matrices (Theorem~\ref{21}).
Therefore, eigenvalue analysis of Hermitian adjacency matrices themselves
seems to be important in order to study properties of the quantum walk
brought by features that each digraph has.
For example, research to determine periodic digraphs
while establishing a method for eigenvalue analysis of
Hermitian adjacency matrices may be interesting.
Also, researching matrices associated with quantum walks
to determine or estimate invariants of digraphs may be worthwhile.


\begin{thebibliography}{99}

\bibitem{AG} Acevedo, OL., Gobron, T.,
\emph{Quantum walks on Cayley graphs},
Journal of Physics A 39(3), 585 (2005).

\bibitem{BDF} Biamonte, De Domenico, M., J., Faccin, M.,
\emph{Complex networks from classical to quantum},
Communications Physics 2.1, (2019) 1--10.

\bibitem{BH} Brouwer, A.E., Haemers, W.H.,
\emph{Spectra of Graphs},
Springer Science \& Business Media (2011).

\bibitem{CP} Chagas, B., Portugal, R.,
\emph{Discrete-Time Quantum Walks on Oriented Graphs},
arXiv:2001.04814v2, (2020).


\bibitem{EHSW} Emms, D., Hancock, E. R., Severini, S., Wilson, R. C.,
\emph{A matrix representation of graphs and its spectrum as a graph invariant},
Electr. J. Combin., 13(1), 2006.

\bibitem{EKOS} Endo, T., Konno, N., Obuse, H., Segawa, E.,
\emph{Sensitivity of quantum walks to boundary of two-dimensional lattices: approaches from the CGMV method and topological phases},
Journal of Physics A: Mathematics and Theoretical 50 (2017) 455302.

\bibitem{FSS} Flamini, F., Sciarrino, F., Spagnolo, N.,
\emph{Photonic quantum information processing: a review},
Reports on Progress in Physics, 82(1), 016001, (2018).

\bibitem{Guo}  Guo, K.,
\emph{Quantum walks on strongly regular graphs}.
Master's thesis, University of Waterloo (2010). 

\bibitem{GG} Godsil, C., Guo, K.,
\emph{Quantum walks on regular graphs and eigenvalues.}
Electr. J. Combin. 18, P165 (2011)

\bibitem{GL} Godsil, C., Lato, S.,
\emph{Perfect State Transfer on Oriented Graphs},
arXiv 2002.04666.
Available at https://arxiv.org/abs/2002.04666, (2020).

\bibitem{GMc} Godsil, C.D., McKay, B.D.,
\emph{Constructing cospectral graphs},
Aequationes Math., 25 (1982), pp. 257--268.

\bibitem{GM} Guo, K., Mohar, B.,
\emph{Hermitian adjacency matrix of digraphs and mixed graphs},
J. Graph Theory (2016).

\bibitem{GS} Gr\"{a}fe, M., Szameit, A.,
\emph{Integrated photonic quantum walks},
Journal of Physics B: Atomic, Molecular and Optical Physics, 53(7), 073001, (2020).

\bibitem{HKSS2013} Higuchi, Y., Konno, N., Sato, I., Segawa, E.,
\emph{A note on the discrete-time evolutions of quantum walk on a graph},
J. Math-for-Industry, 5 (2013B-3): 103--109 (2013).

\bibitem{HKSS2014} Higuchi, Y., Konno, N., Sato, I., Segawa, E.,
\emph{Spectral and asymptotic properties of Grover walks on crystal lattices},
J. Funct. Anal. 267, 4197--4235 (2014).

\bibitem{HKSS2017} Higuchi, Y., Konno, N., Sato, I., Segawa, E.,
\emph{Periodicity of the discrete-time quantum walk on a finite graph},
Interdiscip. Inf. Sci., 23, 75--86 (2017).

\bibitem{HM} Hoyer, S., Meyer, D.A.,
\emph{Faster transport with a directed quantum walk},
Phys. Rev. A79, p. 024307 (2009).

\bibitem{K} Kubota, S.,
\emph{Strongly regular graphs with the same parameters as the symplectic graph},
Sib. Elektron. Mat. Izv. 13, 1314--1338 (2016).

\bibitem{KS2011} 
Konno, N., Sato, I., 
\emph{On the relation between quantum walks and zeta functions}, 
Quantum Information Processing, 11  (2012) pp.341--349. 

\bibitem{KSS2019}
Konno, N., Sato, I., Segawa, E.,
\emph{Phase measurement of quantum walks: application to structure theorem of the positive support of the Grover Walk },
Electr. J. Combin. 26, P2.26 (2019). 

\bibitem{KSTY2018} Kubota, S., Segawa, S., Taniguchi, T., Yoshie, Y.,
\emph{Periodicity of Grover walks on generalized Bethe trees},
Linear Algebra Its Appl., 554, 371--391 (2018).

\bibitem{L} Lato, S.,
\emph{Quantum Walks on Oriented Graphs},
2019. Masters Thesis.

\bibitem{LL} Liu, J., Li, X.,
\emph{Hermitian-adjacency matrices and Hermitian energies of mixed graphs},
Linear Algebra Appl 466 (2015), 182--207.

\bibitem{M} Montanaro, A.,
\emph{Quantum walks on directed graphs},
Quantum Inf. Comp 7, pp. 93--102 (2007).

\bibitem{MW} Manouchehri, K., Wang, J., 
\emph{Physical implementation of quantum walks},
Springer Berlin, (2013).

\bibitem{NP} Neves, L., Puentes, G.,
\emph{Photonic Discrete-time Quantum Walks and Applications},
Entropy, 20(10), (2018) 731.

\bibitem{Portugalbook} Portugal, R., 
\emph{Quantum Walks and Search Algorithm},
Springer (2013).

\bibitem{RenETAL} 
Ren, P., Aleksic, T., Emms, D., Wilson, R. C., Hancock, E. R., 
\emph{Quantum walks, Ihara zeta functions and cospectrality in regular graphs}, 
Quantum Information Processing {\bf 10} (2011) pp.405--417.

\bibitem{S02}
Severini, S.,
\emph{The underlying digraph of a coined quantum random walk},
arXiv preprint quant-ph/0210055, (2002).

\bibitem{S03}
Severini, S.,
\emph{On the Digraph of a Unitary Matrix},
SIAM J. Matrix Anal. Appl.25(1), (2003) pp. 295--300.

\bibitem{Y2017} Yoshie, Y.,
\emph{A characterization of the graphs to induce periodic Grover walk},
Yokohama Math. J., 63, 9--23 (2017).

\bibitem{Y2019} Yoshie, Y.,
\emph{Periodicity of Grover walks on distance-regular graphs},
Graphs Comb., 35 (2019), pp.1305--1321.

\bibitem{Z} Zhan, H.,
\emph{Discrete Quantum Walks on Graphs and Digraphs},
Ph.D. thesis (2018).

\end{thebibliography}
\end{document}